\documentclass{amsart}

\usepackage{tikz}
\usetikzlibrary{matrix,decorations.pathreplacing,calc}

\usepackage{enumitem}

\tikzset{ 1
table/.style={
  matrix of math nodes,
  row sep=-\pgflinewidth,
  column sep=-\pgflinewidth,
  nodes={rectangle,text width=3em,align=center},
  text depth=1.25ex,
  text height=2.5ex,
  nodes in empty cells,
  left delimiter=[,
  right delimiter={]},
  ampersand replacement=\&
}
}

\usepackage{amssymb}
\usepackage{amsmath}
\usepackage{amsfonts}
\usepackage{amsthm}
\usepackage{stmaryrd}
\usepackage[all]{xy}
\usepackage{mathrsfs}
\usepackage{graphicx}
\usepackage{hyperref}
\usepackage{color}
\usepackage{tikz}
\usetikzlibrary{matrix}
\usepackage{amsmath,amscd}
\usepackage{pifont}
\usetikzlibrary{arrows}

\numberwithin{equation}{subsection}
\newtheorem{theorem}[subsection]{Theorem}

\newtheorem{corollary}[subsection]{Corollary}
\newtheorem{lemma}[subsection]{Lemma}
\newtheorem{proposition}[subsection]{Proposition}

\theoremstyle{definition}

\newtheorem{convention}[subsection]{Convention}
\newtheorem{definition}[subsection]{Definition}

\newtheorem{notation}[subsection]{Notation}

\newtheorem{remark}[subsection]{Remark}

\newtheorem{hypothesis}[subsection]{Hypothesis}

\def\calC{\mathcal{C}}

\def\calE{\mathcal{E}}

\def\calO{\mathcal{O}}

\def\calW{\mathcal{W}}

\def\gothm{\mathfrak{m}}

\def\gothS{\mathfrak{S}}
\def\gothT{\mathfrak{T}}

\def\AAA{\mathbb{A}}

\def\CC{\mathbb{C}}
\def\FF{\mathbb{F}}
\def\GG{\mathbb{G}}

\def\PP{\mathbb{P}}
\def\QQ{\mathbb{Q}}

\def\ZZ{\mathbb{Z}}

\def\bfB{\mathbf{B}}

\def\rmT{\mathrm{T}}

\DeclareMathOperator{\Max}{Max}

\newcommand{\rig}{\mathrm{rig}}

\newcommand{\wt}{\mathrm{wt}}

\newcommand{\Zp}{\ZZ_p}
\newcommand{\val}{\mathrm{val}}
\newcommand{\Tr}{\mathrm{Tr}}

\DeclareMathOperator{\GL}{GL}

\newcommand{\adm}{\mathrm{adm}}

\pgfkeys{tikz/mymatrixenv/.style={decoration=brace,every left delimiter/.style={xshift=3pt},every right delimiter/.style={xshift=-3pt}}}
\pgfkeys{tikz/mymatrix/.style={matrix of math nodes,left delimiter=[,right delimiter={]},inner sep=2pt,column sep=1em,row sep=0.5em,nodes={inner sep=0pt}}}
\pgfkeys{tikz/mymatrixbrace/.style={decorate,thick}}
\newcommand\mymatrixbraceoffseth{0.5em}
\newcommand\mymatrixbraceoffsetv{0.2em}

\newcommand*\mymatrixbraceright[4][m]{
    \draw[mymatrixbrace] ($(#1.north west)!(#1-#3-1.south west)!(#1.south west)-(\mymatrixbraceoffseth,0)$)
        -- node[left=2pt] {#4} 
        ($(#1.north west)!(#1-#2-1.north west)!(#1.south west)-(\mymatrixbraceoffseth,0)$);
}

\newcommand*\mymatrixbracetop[4][m]{
    \draw[mymatrixbrace] ($(#1.north west)!(#1-1-#2.north west)!(#1.north east)+(0,\mymatrixbraceoffsetv)$)
        -- node[above=2pt] {#4} 
        ($(#1.north west)!(#1-1-#3.north east)!(#1.north east)+(0,\mymatrixbraceoffsetv)$);
}

\begin{document}\large

\title{Slopes for higher rank Artin--Schreier--Witt Towers} 

\author{Rufei Ren}
\address{Rufei Ren, University of California, Irvine, Department of
Mathematics, 340 Rowland Hall, Irvine, CA 92697}
\email{rufeir@math.uci.edu}
\author{Daqing Wan}
\address{Daqing Wan, University of California, Irvine, Department of
Mathematics, 340 Rowland Hall, Irvine, CA 92697}
\email{dwan@math.uci.edu}
\author{Liang Xiao}
\address{Liang Xiao, University of Connecticut, Department of Mathematics, 341 Mansfield Road, Unit 1009, Storrs, CT 06269-1009}
\email{liang.xiao@uconn.edu}

\author{Myungjun Yu}
\address{Myungjun Yu, University of Michigan, Department of Mathematics,  
2074 East Hall,
530 Church Street,
Ann Arbor, MI 48109-1043}
\email{myungjuy@umich.edu}

\thanks{L.X. is partially supported by Simons Collaboration Grant \#278433 and NSF Grant DMS--1502147.}
\date{\today}

\begin{abstract}
We fix a monic polynomial $\bar f(x) \in \FF_q[x]$ over a finite field of characteristic $p$ of degree relatively prime to $p$, and consider the $\ZZ_{p^{\ell}}$-Artin--Schreier--Witt tower defined by $\bar f(x)$; this is a tower of curves $\cdots \to C_m \to C_{m-1} \to \cdots \to C_0 =\AAA^1$, whose Galois group is canonically isomorphic to $\ZZ_{p^\ell}$, the degree $\ell$ unramified extension of $\ZZ_p$, which is abstractly isomorphic to $(\ZZ_p)^\ell$ as a topological group.
We study the Newton slopes of zeta functions of this tower of curves. This reduces to the study of the Newton slopes of L-functions associated to characters of the Galois group of this tower.  We prove that, when the conductor of the character is large enough, the Newton slopes of the L-function 
asymptotically form a finite union of arithmetic progressions.  As a corollary, we prove the spectral halo property of the spectral variety associated to the $\ZZ_{p^{\ell}}$-Artin--Schreier--Witt tower (over a large subdomain of the weight space). This extends the main result in \cite{Davis-wan-xiao} from rank one case $\ell=1$ to the higher rank case $\ell\geq 1$. 
\end{abstract}

\maketitle

\setcounter{tocdepth}{1}
\tableofcontents

\section{Introduction}
The topic we study in this paper reflects interests from two related areas. We shall first introduce our theorem from the $p$-adic and Iwasawa theoretic perspective of $L$-functions of varieties, and then explain the (philosophical) implication on spectral halo of eigenvarieties.

For a positive integer $\ell$, a $\ZZ_p^{\ell}$-Witt tower over a finite field $\FF_q$ of characteristic $p$ is a sequence of finite \'etale Galois covers over $\FF_q$, \[\cdots \to C_m \to \cdots \to C_1 \to C_0= \AAA^1,\] whose total Galois group is isomorphic to $\ZZ_p^{\ell}$. The integer $\ell$ is called the 
\emph{rank} of the tower. All such Witt towers, uncountably many,  can be constructed explicitly from Witt vectors, and their genera 
can be read off from an explicit formula, see \cite{KW}.  
A main interest in arithmetic geometry is to understand the zeros of the zeta-functions of the curves $C_m$ over $\FF_q$. In the context of Witt towers and the spirit of Iwasawa theory, a natural question is: what are the $p$-adic valuations (slopes) of the zeros of the zeta-function of $C_m$, especially what is the asymptotic behavior as $m \to \infty$? This is an emerging new field of study, which is expected to be quite fruitful and yet rather complicated in general, as there are too many Witt towers and most of them behave very badly. In order for the valuation sequence to have a strong stable property as $m$ grows, it is reasonable (and necessary) to assume that the genus sequence 
has a stable property. Fortunately, Witt towers with a stable genus formula can be classified, and this is recently 
done in  \cite{KW}. It is then natural to investigate the deeper slope stable property for the zeta function sequence 
of a genus stable Witt tower. 

The first nontrivial case is when the tower is defined by the Teichm\"uller lift of a polynomial over $\FF_q$ (see the next paragraph), called the Artin--Schreier--Witt tower, which does satisfy the genus stable property. 
When the Artin--Schreier--Witt tower has the Galois group $\ZZ_p$ (rank one case), the slope stability  question has been successfully answered  in  \cite{Davis-wan-xiao}, where it is shown that the valuations of the zeros  are given by a finite union of arithmetic progressions. This implies a strong  stable property for the slopes when $m \to \infty$.  Our goal of this paper is to generalize the results in \cite{Davis-wan-xiao} to the higher rank case, that is, to Artin--Schreier--Witt towers 
whose Galois groups are canonically identified with $\ZZ_{p^\ell}$ which is the unramified extension of $\ZZ_p$ of degree $\ell$, by a suitable adaptation of the methods in \cite{Davis-wan-xiao}.  The argument turns out to be more difficult because the space of characters is now multi-dimensional (see the discussion after Theorem~\ref{mainforC}).

Let us be more precise.
Fix a prime number $p$. Let $\FF_q$ be a finite extension of $\FF_p$ of degree $a$ so that $q =p ^a$. Let $\ell$ be an integer which divides $a$.
For an element $\bar b \in \overline{\FF}_q^\times$, let $\omega(\bar b)$ denote its Teichm\"uller lift in $\ZZ_q$ (the unramified extension of $\ZZ_p$ with residue field $\FF_q$); we put $\omega(0) = 0$.
Let $\sigma$ denote (the lift of) the arithmetic $p$-Frobenius on $\FF_q$ and $\ZZ_q$.

We fix a monic polynomial $\bar f(x) = x^d + \bar a_{d-1}x^{d-1} + \cdots +  \bar a_0 \in \FF_q[x]$ whose degree $d$ is \emph{not} divisible by $p$. 
We write $\bar a_d=1$, and $a_i: = \omega(\bar a_i)$ for $i=0, \dots, d$. Let $f(x)$ denote the polynomial $ x^d + a_{d-1}x^{d-1} + \cdots + a_0 \in \ZZ_q[x]$, called the \emph{Teichm\"uller lift} of the polynomial $\bar f(x)$.
The $\ZZ_{p^{\ell}}$-\emph{Artin--Schreier--Witt} tower associated to $f(x)$ is the sequence of curves $C_m$ over $\FF_q$ defined by 
\[
C_m: \quad \underline y_m^{F^\ell} - \underline y_m =  \sum_{i=0}^d(\bar a_ix^i, 0, 0,\dots)_m,
\]
where $\underline y_m = (y_1, y_2, \dots, y_m)$ are viewed as Witt vectors of length $m$, and $\bullet^F$ means raising each Witt coordinate to the $p$-th power.
In explicit terms, this means that $C_1$ is the usual Artin--Schreier curve given by $y_1^{p^\ell} - y_1 = \bar f(x)$, and $C_2$ is the curve above $C_1$ given by an additional equation (over $\FF_q$)
\[
y_2^{p^\ell} - y_2 + \frac{y_1^{p^{\ell+1}}-y_1^p - (y_1^{p^\ell} - y_1)^p}{p} =  \frac{f^\sigma(x^p) - (f(x))^p}{p} \quad \bmod p,\]
where $\sigma$ is the Frobenius automorphism and $ f^\sigma(x) : = x^d +  \sigma(a_{d-1}) x^{d-1} + \dots + \sigma(a_0)$.

The Galois group of the tower may be identified with $\ZZ_{p^\ell}$, such that $a\in \ZZ_{p^\ell}$ sends $\underline{y}_m$ to $\underline{y}_m+\underline{a}_m$, 
where $\underline{a}_m$ denotes the $m$-th truncated Witt vector of $a$.  Each curve $C_m$ has a zeta function defined by
\[
Z(C_m,s) = \exp \left(\sum_{k\geq 1} \frac{s^k}{k} \cdot \# C_m(\FF_{q^k}) \right) =\frac{P(C_m,s)}{1-qs},  
\]
where $P(C_m, s) \in 1+s\ZZ[s]$ is a polynomial of degree $2g(C_m)$, pure of $q$-weight $1$, and $g(C_m)$ denotes the genus of $C_m$.   

Write $\CC_p$ for the completion of an algebraic closure of $\QQ_p$, and let $\calO_{\CC_p}$ denote its valuation ring with maximal ideal $\gothm_{\CC_p}$.
Using the Galois group of $C_m$ over $\AAA^1$, we may factor $Z(C_m,s)$ into a product of L-functions: 

\[Z(C_m,s)=\prod\limits_{\chi:~ \ZZ_{p^{\ell}}/p^m\ZZ_{p^{\ell}} \to \CC_p^\times }L_f(\chi, s),\]
where for each character $\chi$, $L_f(\chi, s)$ is the $L$-function on $\AAA_{\FF_q}^1$ given by

%
\begin{equation}
L_f(\chi, s)  = \prod\limits_{x \in |\AAA^1|}
\frac{1}{1-\chi\Big(\Tr_{\QQ_{q^{\deg(x)}}/\QQ_{p^\ell}}\big( f(\omega( x))\big)\Big)s^{\deg(x)}}\,,
\end{equation}
where $ |\AAA^1|$ denotes the set of closed points of $\AAA^1_{\FF_{q}}$ and $\omega(x)$ denotes the Teichm\"uller lift of  any of the conjugate geometric points in the closed point $x$. For $\chi=1$, the L-function $L_f(1,s)$ is simply the trivial factor $1/(1-qs)$, which is the 
zeta function of the affine line. 

The goal of this paper is to understand the $p$-adic valuation of the zeros of these $L$-functions for all non-trivial finite characters $\chi$. For this 
purpose, we will also need to consider the characters which are not finite and put them in a family. 
In this paper, all characters $ \chi: \ZZ_{p^\ell} \to \CC_p^\times$ are assumed to be continuous. For a finite character $\chi$, let $m_{\chi}$ be the nonnegative integer so that the image of ${\chi}$ has cardinality $p^{m_{\chi}}$; we call $m_{\chi}$ the \emph{conductor} of ${\chi}$. Our normalization on Newton polygons is as follows: given a valuation ring $R$ and an element $\varpi$ of positive valuation, the $\varpi$-adic Newton polygon of a power series $c_0 + c_1s + \cdots \in R\llbracket s\rrbracket$ is the lower convex hull of the points $(k,\val_\varpi(c_k))$ ($k \in \ZZ_{\geq 0}$), where the valuation $\val_\varpi(-)$ is normalized so that $\val_\varpi(\varpi)=1$.

\begin{theorem}[Main Theorem]\label{mainforL}
For any nontrivial finite character $\chi$ with conductor $m_\chi$, $L_f(\chi, s)$ is a polynomial of degree $dp^{m_\chi-1}-1$. Write 
$$L_f(\chi, s) = \sum\limits_{k=0}^{dp^{m_\chi-1}-1} c_ks^k.$$
We have the following.
\begin{enumerate}[label=(\roman*)]
\item For any $0 < n\leq p^{m_\chi-1}$, we have $\val_q(c_{nd-1})=\frac{n(nd-1)}{2p^{m_\chi-1}}$ and $\val_q(c_{nd})=\frac{n(nd+1)}{2p^{m_\chi-1}}$.
\item For any $0 < n\leq p^{m_\chi-1}$, the $q$-adic Newton polygon of $L_f(\chi, s)$ passes through the points $\big(nd-1,\frac{n(nd-1)}{2p^{m_\chi-1}}\big)$ and  $\big(nd,\frac{n(nd+1)}{2p^{m_\chi-1}}\big)$.
\item The $q$-adic Newton polygon of $L_f(\chi, s)$ has slopes (in increasing order) 
\[\bigcup_{i=1}^{p^{m_\chi-1}} \{\alpha_{i1},\alpha_{i2},\dots,\alpha_{id}\}-\{0\}, 
\]
where \[\begin{cases}
\alpha_{ij}=\frac{i-1}{p^{m_\chi-1}}& \mathrm{for}\ j=1,\\
\frac{i-1}{p^{m_\chi-1}}< \alpha_{ij}< \frac{i}{p^{m_\chi-1}}& \mathrm{for}\ j>1.
		\end{cases}\]
	\end{enumerate}
\end{theorem}

\begin{remark}
We do not know how to get the arithmetic progression property as in \cite{Davis-wan-xiao}, which is uniform in $\chi$ (depending only 
on the large conductor $m_{\chi}$, not on the choice of $\chi$ with the given conductor $m_\chi$). However,
for $j=1$, the slopes $\alpha_{i1} = \frac{i-1}{p^{m_\chi-1}}$ do form an arithmetic progression, which depends only on the conductor 
$m_{\chi}$. For any fixed $j>1$, part (iii) only proves that the slopes $\alpha_{ij}$ are approximately an arithmetic progression.  

If we restrict to those characters $\chi$ that factor through a fixed quotient $\eta: \ZZ_{p^\ell} \to \ZZ_p$, then the slopes $\alpha_{ij}$ form a union of finitely many  arithmetic progressions (independent of the character $\chi$ but a priori depending on the quotient $\eta$), as the problem reduces to the case of usual $\ZZ_p$-tower but with non-Teichm\"uller polynomials considered in \cite{li}.
It is unclear whether these arithmetic progressions depend on the choice of the quotient $\eta: \ZZ_{p^\ell} \to \ZZ_p$.	
	\end{remark}

It would be more convenient for us to consider the $p$-adic function defined by 
\begin{equation}\label{CL}
C_f^*(\chi,s)= L_f^*(\chi,s) L_f^*(\chi,qs)L_f^*(\chi,q^2s) \cdots,
\end{equation}
where $L_f^*(\chi,s) := (1 - \chi(\Tr_{\QQ_{q}/\QQ_{p^{\ell}}}(f(0)))s)L_f(\chi, s)$ is the L-function of $\chi$ over the torus $\mathbb{G}_m=\mathbb{A}^1 -\{0\}$.

From $C_f^*(\chi,s)$, one may recover $L_f^*(\chi, s)$ as $L_f^*(\chi, s) = \frac{C_f^*(\chi,s)}{C_f^*(\chi,qs)}$.
Hence Theorem \ref{mainforL} is essentially a corollary of the following.
\begin{theorem}\label{coincide}
Given a nontrivial finite character $\chi$ with conductor $m_\chi$, write 
$$C_f^*(\chi, s)=\sum\limits_{k=0}^\infty w_k(\chi)s^k.$$ Then for all $k\geq 0$, we have 
$$ \val_q(w_k(\chi))\geq \frac{k(k-1)}{2dp^{m_\chi-1}} \quad \textrm{and},
$$
$$\textrm{when }k=nd\textrm{ or }nd+1, \quad  \val_q(w_k(\chi))=\frac{k(k-1)}{2dp^{m_\chi-1}}.$$  
In particular, the $q$-adic Newton polygon of $C_f^*(\chi, s)$ passes through the points $(nd,\frac{n(nd-1)}{2p^{m_\chi-1}})$ and $(nd+1,\frac{n(nd+1)}{2p^{m_\chi-1}})$ for all $n\geq 0$.
	
\end{theorem}

We will show that Theorems~\ref{mainforL} and \ref{coincide} follow from Theorem~\ref{mainforC} below, in \ref{S:Thm 1.4=>1.1 and 1.3}.

%
%
%
%
%
%
%
%
%
%
%
%

To effectively prove Theorem \ref{coincide}, it is important to consider all characters in a big family.
We fix a basis $\{ c_1,\dots, c_\ell\}$ of $\ZZ_{p^\ell}$ as a free $\ZZ_p$-module; we write $\bar c_j = c_j \bmod p$ for each $j$.
 The Galois group $\ZZ_{p^\ell}$ of the tower can be identified with $\ZZ_p^\ell$ explicitly as 
\[\xymatrix@R=0pt{
\qquad&\ZZ_{p^\ell} \ar[r]^{\cong}&  \Zp^\ell\\
& x \ar@{|->}[r] & \big(\Tr_{\QQ_{p^\ell}/\QQ_p}(xc_1),\dots,\Tr_{\QQ_{p^\ell}/\QQ_p}(xc_\ell)\big).
} \]
We consider the \emph{universal character} of $\ZZ_{p^\ell}$:
\[
\xymatrix@R=0pt@C=10pt{
\chi_\mathrm{univ}: & \ZZ_{p^\ell} \ar[rr] && \ZZ_p\llbracket \underline T \rrbracket ^\times :=  \ZZ_p \llbracket T_1, \dots, T_\ell\rrbracket^\times
\\
&x \ar@{|->}[rr] && (1+T_1)^{\Tr_{\QQ_{p^\ell}/\QQ_p}(xc_1)}
 \cdots (1+T_\ell)^{ \Tr_{\QQ_{p^\ell}/\QQ_p}(xc_\ell)}.
}
\]
(When $\ell=1$, we simply write $T$ for $\underline T$.)
Any continuous character $\chi: \ZZ_{p^\ell} \to \CC_p^\times$ can be recovered from $\chi_\mathrm{univ}$ by evaluating each $T_j$ at $\chi(c_j^*)-1$, where $c_1^*, \dots, c_\ell^* \in \ZZ_{p^\ell}$ are elements such that $\Tr_{\QQ_{p^\ell}/\QQ_p}(c_i^*c_j)$ is equal to $1$ if $i=j$ and is equal to $0$ if $i\neq j$.

Similar to the finite character case, we will define in Section \ref{section 2} a power series
\[
C_f^*(\chi_\mathrm{univ}, s) = C_f^*(\underline{T},s)= 1+ w_1(\underline{T})s + w_2 (\underline{T})s^2+ \cdots \in 1+ s\ZZ_p\llbracket \underline{T}\rrbracket \llbracket s \rrbracket,\]
for the universal character $\chi_\mathrm{univ}$. This power series interpolates $C_f^*(\chi,s)$ for all (finite) characters $\chi: \ZZ_{p^\ell} \to \CC_p^\times$ via the formula
\[
C_f^*(\chi, s) = C_f^*(\underline T, s)|_{T_j = \chi(c_j^*)-1 \textrm{ for all }j}.
\]

\begin{theorem}\label{mainforC}
Let $I$ denote the ideal $ (T_1, \dots, T_\ell) \subseteq \ZZ_p\llbracket \underline T \rrbracket$.
For $k \in \ZZ_{\geq 0}$, we put $
\lambda_k = \frac{ ak(k-1)(p-1)}{2d}
$.  Then we have the following.

\emph{(1)} For any $k>0$, we have
\begin{equation}
\label{E:hodge bound}
w_k(\underline T) \in I^{\lceil \lambda_k \rceil}.
\end{equation}

\emph{(2)} When $k=nd$ or $nd+1$, we have\footnote{Under the hypothesis $p\nmid d$, $\lambda_{nd}/\ell =  \frac{an(nd-1)(p-1)}{2\ell}$ and $\lambda_{nd+1}/\ell = \lambda_{nd}/\ell + \frac{an(p-1)}\ell$ are always integers.}
\begin{equation}\label{leading term}
w_k(\underline{T}) \equiv u_k\cdot  \gothS(\underline T)^{\lambda_k/\ell} \ \bmod (pI^{\lambda_k} + I^{\lambda_k+1})  
\end{equation}
for some unit $u_k \in \ZZ_p$, where $\gothS(\underline T)$ is the following polynomial
\begin{equation}
\label{E:gothS}
\gothS(\underline T): = 
\prod\limits_{i=1}^\ell\Big(\sum\limits_{j=1}^\ell \sigma^i(c_j)T_j\Big).
\end{equation}
\end{theorem}

Theorem~\ref{mainforC} is the main technical result of this paper. Part (1) is proved at the end of Section~\ref{Sec:Hodge bound}; part (2) is proved at the end of Section~\ref{Sec:main proof}, relying on the key Theorem~\ref{T:key technical theorem}.

Let us now explain the philosophical meaning of Theorem~\ref{mainforC}.
The first estimate \eqref{E:hodge bound} uses a standard argument to establish certain Hodge bound.
It implies, for example, when $k=nd$ or $nd+1$, the ``leading term" (if nonzero) of $w_k(\underline T)$ must be a homogeneous polynomial of degree $\lambda_k$ in $\underline T$.
\begin{itemize}
\item[(i)]
When $\ell=1$, this leading term has to be a monomial in $T$; so specializing to any continuous non-trivial  character $\chi$ of $\ZZ_p$, this ``leading term" (if its coefficient is a $p$-adic unit) is also the ``leading term" of $w_k(\chi)$. Theorem~\ref{mainforC}(2) is proved in \cite[Proposition~3.4]{Davis-wan-xiao}, which is the key of the proof of \cite[Theorem~1.2]{Davis-wan-xiao}. 
\item[(ii)]
In clear contrast, when $\ell>1$, this ``leading term", even if its coefficients are $p$-adic units, may not continue to have smaller valuation than higher degree terms after certain specialization.
In particular, the na\"ive generalization of Theorem~\ref{coincide} to \emph{all} non-trivial characters of $\ZZ_{p^\ell}$ is \emph{false}.
It is thus of crucial importance to understand: what does the ``leading term" of $w_k(\underline T)$ look like?
This is exactly answered by \eqref{leading term}, which shows that the ``leading term" of $w_k(\underline T)$ is, up to a $p$-adic unit, a power of a particular polynomial $\gothS(\underline T)$ \emph{independent of $k$ and of the Teichm\"uller polynomial $f$}.

\end{itemize}

We also point out that the polynomial $\gothS(\underline T)$ modulo $p I^\ell+ I^{\ell+1}$ is canonically independent of the choice of the basis $\{c_1, \dots, c_\ell\}$ (Lemma~\ref{L:independence}).  Moreover, $\gothS(\underline T)$ is in some sense ``elliptic" as its zero \emph{avoids} all the evaluations of the $T_j$'s corresponding to \emph{finite} continuous non-trivial characters of $\ZZ_{p^\ell}$ (Lemma~\ref{L:admissible locus contains finite characters}).

\medskip
While Theorem~\ref{mainforC} is known when $\ell =1$ by \cite{Davis-wan-xiao}, its proof for general $\ell$ is quite different.
The idea lies in a careful study of the matrix whose characteristic power series gives rise to $C_f^*(\underline T, s)$.
We do this in two steps. The first step is to show that the leading term of $w_k(\underline T)$ comes from the determinant of the upper left $k\times k$-submatrix. The second step is to show that the determinant of the mod $p$ reduction of the upper left $k\times k$-submatrix is ``independent of $\ell$", in the sense that it is the same matrix for the $\ell=1$ case except replacing $T = T_1$ by the polynomial $\sum_{j=1}^\ell c_j T_j$; see Theorem~\ref{T:key technical theorem}.\footnote{The product over the Frobenius twists of $c_j$ is a result of the setup of the Dwork's trace formula; see Corollary~\ref{determinant}.}
In this way, we reduce the proof of Theorem~\ref{mainforC} for general $\ell$ to the known case of $\ell=1$.



\subsection{Analogy with the Igusa tower of (Hilbert) modular varieties}
An important philosophical implication of Theorem~\ref{mainforC} is through the close analogy between the $\ZZ_p$-Artin--Schreier--Witt tower and the Igusa tower  of  modular curves:
\begin{itemize}
\item
the Galois group $\ZZ_p$ of the $\ZZ_p$-Artin--Schreier--Witt tower is the additive version of the Galois group $\ZZ_{p}^\times$ of the Igusa tower of modular curves,
\item
the big Banach module $\widetilde {\mathbf{B}}$ in \eqref{E:widetilde B} is analogous to the space of overconvergent modular forms,
\item
the linear operator $\psi$ defined in \eqref{E:psi} is analogous to the $U_p$-operator, and
\item
the power series $C_f^*(T, s)$ is analogous to the characteristic power series of $U_p$.
\end{itemize}

Inspired by this analogy, we define the \emph{Artin--Schreier--Witt eigenvariety} to be the zero locus\footnote{We will explain in Section~\ref{Sec:ASW eigenvarieties} the meaning of ``zero locus."} of the universal multi-variable power series $C_f^*(\underline{T},s)$ inside $\GG_m^\rig \times (\calW - \{\underline 0\})$,\footnote{For Artin--Schreier--Witt tower, the trivial character behaves slightly differently.} where $\calW$ is the \emph{weight space} $\Max (\ZZ_p\llbracket \underline T\rrbracket[\frac 1p] )$; explicitly, $\calW$ is the $\ell$-dimensional open unit polydisk.\footnote{We refer to \ref{S:weight space} for the subtleties in defining $\calW$ and $\calW-\{0\}$.} As shown in the diagram below, this eigenvariety $\calE_f$ admits a \emph{weight map} $\wt$ to the weight space, and an ``$a_p$-map" to $\GG_m^\rig$ remembering the values of $s^{-1}$. The $p$-adic valuation of the image of the ``$a_p$-map" is called the \emph{slope} of the point.
\[
\xymatrix{
\calE_f\ar@/^10pt/[rr]^{\mathrm{slope}}\ar[d]_{\textrm{wt}} \ar[r]_-{a_p} & \mathbb{C}_p^\times \ar[r]_{\mathrm{val.}} & \mathbb Q.
\\
\calW-\{\underline 0\}
}
\]
This gives a full picture analogous to the case of eigencurves for overconvergent modular forms, or more generally eigenvarieties for overconvergent Hilbert modular forms.

A key component of the analogy is that the proof of the decomposition of the $\ZZ_p$-Artin--Schreier--Witt eigencurve (\cite[Theorem~4.2]{Davis-wan-xiao}) is very similar to the proof of the decomposition of the Coleman--Mazur eigencurve over the boundary of weight space (\cite[Theorem~1.3]{liu-wan-xiao}), where the Hodge estimate (see Definition~\ref{D:twisted incremental}) is analogous to certain variant of the Hodge estimate in \cite[Proposition~3.12(1)]{liu-wan-xiao}, and  the numerics provided by the Poincar\'e duality of L-functions corresponds to the numerics given by the Atkin--Lehner involution. The only difference is that the Hodge lower bound in \cite{liu-wan-xiao} is obtained by a slightly different mechanism.\footnote{In \cite{liu-wan-xiao}, we looked at the Betti realization instead of the de Rham realization to circumvent the technical difficulties caused by the geometry of the base modular curve. This is in fact a crucial point. We do not know how to prove spectral halo type results for overconvergent modular forms in the de Rham setup.}

A caveat for the reader is that, we think, the analogy is entirely on the philosophical level, one cannot deduce theorems about overconvergent Hilbert modular forms directly from the analogous results for Artin--Schreier--Witt towers.

The state-of-the-art technique on the study of spectral halo (based on \cite{liu-wan-xiao}) is intrinsic to $\GL_2(\QQ_p)$.\footnote{F. Andreatta, Iovita, and V. Pilloni \cite{AIP} defined certain extension of the Hilbert eigenvariety to the ``adic boundary" of the weight space. Unfortunately, they could not prove the analogous spectral halo result.
Building on this and \cite{liu-wan-xiao},  C. Johansson and J. Newton \cite{johansson-newton} recently generalized the Hodge estimate of \cite{liu-wan-xiao}. Still, the na\"ive generalization of the numerical coincidence no longer holds, and much less is known regarding to this extension, compare to the $\GL_2(\QQ_p)$-case.} It is natural to try to extend \cite{liu-wan-xiao} beyond this case, say to $\GL_2(\QQ_{p^\ell})$, that is to study the eigenvarieties associated to overconvergent Hilbert modular forms, for a totally real field $F$ of degree $\ell$ in which $p$ is totally inert. In this case, one can interpret the question in terms of the Igusa tower of $\ell$-dimensional Hilbert modular varieties with Galois group $\ZZ_{p^\ell}^\times$.  Via the analogy, this should correspond to the $\ZZ_{p^\ell}$-Artin--Schreier--Witt tower of varieties over the $\ell$-dimensional base $(\GG_m)^\ell$ for $\ell>1$.
There are now two distinct types of generalizations we encounter:
\begin{itemize}
\item[(a)] the weight space has become multi-dimensional, and
\item[(b)] the base of the variety has become multi-dimensional.
\end{itemize}
Interestingly, for automorphic eigenvarieties, these two generalizations appear simultaneously, whereas on the Artin--Schreier--Witt side, we can tackle them one at a time.

This paper addresses the generalization (a). The solution we propose is the following: it might be too much to ask for a decomposition of the eigenvariety over the entire weight space such that the slopes on each component are determined by the weight map, exactly because of the issue explained in (ii) after Theorem~\ref{mainforC}.\footnote{An alternative way to explain this is: the ``adic boundary" of the weight space is (non-canonically isomorphic to) $\PP^{\ell-1}_{\FF_p}$; so when $\ell=1$, there is only one direction approaching the boundary. But when $\ell>1$, we may have to give up on some ``bad directions" approaching the boundary. Theorem~\ref{mainforC} says that the bad direction is exactly the hypersurface defined by  the polynomial $\prod\limits_{i=1}^\ell\big(\sum\limits_{j=1}^\ell \sigma^i( c_j)T_j\big)$ mod $p$.}
Instead, we study a subspace of $\calW$, the \emph{admissible locus}, defined by
\[
\calW^\adm :=\big\{ \underline t \in \calW(\CC_p) - \{\underline 0\} \; |\; \val_q(\gothS(\underline t)) = \ell \cdot \min\{\val_q(t_1), \dots, \val_q(t_\ell)\} \big\}.
\]
This is an increasing union of affinoid subdomains of $\calW-\{0\}$, which is independent of the polynomial $f$, and is canonically independent of the choice of basis $\{c_1, \dots, c_\ell\}$ (Corollary~\ref{C:admissible locus independent}).
Moreover, $\calW^\adm$ contains all points corresponding to \emph{finite} non-trivial characters of $\ZZ_{p^\ell}$ (Lemma~\ref{L:admissible locus contains finite characters}).

One corollary of Theorem~\ref{mainforC} is the following.
\begin{theorem}
\label{T:eigenvariety decomposition}
Put $\calE_f^\adm: = \wt^{-1}(\calW^\adm)$. Then $\calE_f^\adm$ is the disjoint union 
\[
\calE_f^{\adm}=X_{0}\coprod X_{(0,1)}\coprod X_{1}\coprod X_{(1,2)}\coprod \cdots, \]
of infinitely many rigid subspaces, such that for each interval $J = [n,n]$ or $(n,n+1)$,
\begin{itemize}
\item
the map $\wt: X_J \to \calW^\adm$ is finite and flat, and
\item
for each point $x \in X_J$, we have
\[
\frac{\val_q(a_p(x))}{\val_q(\gothS(\wt(x)))} \in \frac{a(p-1)}{\ell}\cdot J.
\]
\end{itemize}
\end{theorem}
This is Theorem~\ref{T:eigencurve theorem}.

\begin{remark}
One may interpret Theorem~\ref{T:eigenvariety decomposition} as the pull-back of the following diagram:
\[
\xymatrix{
\textrm{decomposition pattern of }\calE_f^\adm \ar[r] \ar[d] & \textrm{decomposition pattern of }\calE_{f, \ZZ_p}
\ar[d]
\\
\calW^\mathrm{adm} \ar[r]^{\underline T \longmapsto T = \gothS(\underline T)}
&
\calW_{\ZZ_p}-\{0\},
}
\]
where the right hand side is the corresponding theorem (\cite[Theorem~4.2]{Davis-wan-xiao}) for the case $\ell=1$.
\end{remark}

\subsection*{Roadmap of the paper}
In Section~\ref{Sec:weight space}, we give several basic facts regarding the polynomial $\gothS(\underline T)$,
and show that Theorems~\ref{mainforL} and \ref{coincide} follow from Theorem~\ref{mainforC}.
Starting from Section~\ref{section 2}, we use another set of variables $\underline \pi$ instead of $\underline T$. We define the characteristic power series $C_f^*(\underline \pi,s)$ in Section~\ref{section 2}, and give a lower bound for its $I$-adic Newton polygon in Section~\ref{Sec:Hodge bound}.
Section~\ref{Sec:main proof} is devoted to the proof of Theorem~\ref{mainforC}, by showing that its validity is independent of $\ell$ and hence reduce to the known case $\ell=1$.
Section~\ref{Sec:ASW eigenvarieties} interprets everything in the language of eigenvarieties.  In the appendix, we include several errata for the paper \cite{Davis-wan-xiao}.

\subsection*{Acknowledgments}
We thank John Bergdall, C. Douglas Haessig, Kiran Kedlaya, and Hui June Zhu for their interests in the paper and interesting discussions. We particularly thank the anonymous referee for many useful comments to improve the presentation of the paper.
	
\section{Weight space}
\label{Sec:weight space}

We collect some basic facts regarding the weight space and characters of $\ZZ_{p^\ell}$.
\begin{lemma}
\label{L:independence}
\emph{(1)} The ideal $I = (T_1, \dots, T_\ell) \subseteq \ZZ_p \llbracket \underline T \rrbracket \cong \ZZ_p \llbracket \ZZ_{p^\ell} \rrbracket$ is canonically independent of the choice of the basis $\{  c_1, \dots,  c_\ell\}$.

\emph{(2)} The polynomial $\gothS(\underline T) \bmod pI^\ell + I^{\ell+1}$ is independent of the choice of the basis $\{  c_1, \dots,  c_\ell\}$.
\end{lemma}
\begin{proof}
(1) Note that a change of basis of $\ZZ_{p^\ell}$ over $\ZZ_p$ results in a change of variables of $\{T_1,\dots, T_\ell\}$ in a way that $\chi_{\mathrm{univ}}$ is well-defined. In fact, $I$ is the augmentation ideal, or equivalently the kernel of $\ZZ_p\llbracket \ZZ_{p^\ell}\rrbracket \twoheadrightarrow \ZZ_p$. So it is canonically independent of the choice of the basis.

(2)
The group of all possible change of basis matrices $\GL_\ell(\ZZ_p)$ is generated by the following three types:
\begin{itemize}
\item[(a)] only swapping $ c_i$ with $ c_j$;
\item[(b)] for a unique fixed $i$, scaling $ c_i$ to $ u_i c_i$ for $ u_i \in \ZZ_p^\times$;
\item[(c)] only changing $ c_1$ to $ c_1+  u c_2$ for $ u \in \ZZ_p$.
\end{itemize}
It suffices to check the invariance of $\gothS(\underline T) \bmod pI^\ell + I^{\ell+1}$ under these three changes of basis.
Case (a) will result in swapping $T_i$ with $T_j$. The invariance of $\gothS(\underline T)$ follows from the definition.
Case (b) will result in changing $T_i$ to $(1+T_i)^{ u_i^{-1} }-1$. The invariance of $\gothS(\underline T) \bmod pI^\ell+I^{\ell+1}$ of this change of basis follows from the congruence
\[
c_iT_i \equiv u_i c_i \big( (1+T_i)^{ u_i^{-1}} -1 \big) \quad \bmod pI + I^2.
\]
Case (c) will result in changing $T_2$ to $(1+T_2)(1+T_1)^{-u}-1$, and keeping all the other variables unchanged.
Then the invariance of $\gothS(\underline T) \bmod pI^\ell+ I^{\ell+1}$ of $\gothS(\underline T)$ follows from the congruence
\[
c_1T_1 + c_2T_2 \equiv (c_1+uc_2) T_1 + c_2 \big((1+T_2)(1+T_1)^{-u}-1\big)\quad  \bmod pI + I^2. \qedhere
\]
\end{proof}

\begin{remark}
\label{R:independence}
In view of Lemma~\ref{L:independence}, the validity of Theorem~\ref{mainforC} is independent of $\{c_1, \dots, c_\ell\}$; so it suffices to prove it for a particular choice of basis $\{c_1, \dots, c_\ell\}$.
\end{remark}

\subsection{Weight space}
\label{S:weight space}
Using the variables $T_1, \dots, T_\ell$, we can explicitly present the weight space as
\[
\calW : = \Max\big(\ZZ_p\llbracket \ZZ_{p^\ell} \rrbracket[\tfrac 1p]\big)
= \big\{ (t_1, \dots, t_\ell) \in \CC_p \; \big|\; \val_q(t_j) > 0 \textrm{ for all }j\big\}.
\]
It is the rigid analytic space associated to the formal scheme $\ZZ_p\llbracket \underline T \rrbracket$ \`a la Raynaud (see e.g. \cite[page 133--134]{mumford} or \cite[\S 0.2]{berthelot}). Explicitly, it is the increasing union of affinoids given by closed polydisks of radius $r$ approaching $1$.
The complement $\calW -\{0\}$ is understood as the rigid analytic space given by the increasing union of polyannuli
\[
\big\{ (t_1, \dots, t_\ell) \in \CC_p \; \big|\; r< \val_q(t_j) <s \textrm{ for all }j \big\}
\]
with $r,s \in \QQ$, $r$ approaching $0^+$, and $s$ approaching $\infty$.

Since $\gothS(\underline T)$ is a homogeneous polynomial of degree $\ell$, we have
\[
\val_q(\gothS(\underline t)) \geq \ell \cdot \min \big\{ \val_q(t_1), \dots, \val_q(t_\ell)\big\}.
\]
Our theory will apply to the case when the above inequality is an equality, namely over the \emph{admissible locus}
\[
\calW^\adm: = \big\{(t_1, \dots, t_\ell) \in \calW - \{\underline 0\} \; \big| \;\val_q(\gothS(\underline t)) = \ell \cdot \min\{\val_q(t_1), \dots, \val_q(t_\ell)\}
\big\}.
\]
This is a rigid analytic subspace of $\calW$. Over each affinoid subdomain $U=\Max(A)$ of $\calW$, $U \cap \calW^\adm$ is the affinoid subdomain given by
\[
\Max \Big( A \big\langle\tfrac{T_1^\ell}{\gothS(\underline T)},\dots, \tfrac{T_\ell^\ell}{\gothS(\underline T)} \big\rangle \Big).
\]

\begin{corollary}
\label{C:admissible locus independent}
The admissible locus $\calW^\adm \subset \calW$ is independent of the choice of the basis $\{c_1, \dots, c_\ell\}$.
\end{corollary}
\begin{proof}
This follows from Lemma~\ref{L:independence}(2) and the definition of the admissible locus.
\end{proof}

\begin{lemma}
\label{L:admissible locus contains finite characters} The coordinate of a continuous character $\chi: \ZZ_{p^\ell} \to \CC_p^\times$  on the weight space is given by $t_{j, \chi}: = \chi(c_j^*) - 1 $ for $j=1, \dots, \ell$. When $\chi$ is a \emph{finite and non-trivial} character, the corresponding point lies on the admissible locus $\calW^\adm$.
\end{lemma}
\begin{proof}
The coordinate of $\chi$ is clearly as given. Let $m_\chi$ denote the conductor of $\chi$, so that the image of $\chi$ lies in $\ZZ_p[\zeta_{p^{m_\chi}}]$. In particular, each $t_{j,\chi} \in \ZZ_p[\zeta_{p^{m_\chi}}]$.

Note that $c_1, \dots, c_\ell$ form a basis of $\ZZ_{p^\ell}$ over $\ZZ_p$. So they also form an \emph{orthonormal basis} of $\ZZ_{p^\ell}[\zeta_{p^{m_\chi}}]$ over $\ZZ_p[\zeta_{p^{m_\chi}}]$. It then follows that
\[
\val_q \Big( \sum_{j=1}^\ell c_j t_{j,\chi} \Big) = \min \big\{ \val_q(t_{1, \chi}), \dots, \val_q(t_{\ell, \chi}) \big\}.
\]
Taking the norm from $\ZZ_{p^\ell}[\zeta_{p^{m_\chi}}]$ to $\ZZ_p[\zeta_{p^{m_\chi}}]$ shows that
\[
\val_q\big(\gothS(t_{1, \chi},\dots, t_{\ell, \chi})\big) = \ell \cdot \min \big\{\val_q(t_{1, \chi}), \dots, \val_q(t_{\ell, \chi})
\big\}.
\]
This means that the point corresponding to $\chi$ lies in $\calW^\adm$.
\end{proof}

\subsection{Proof of Theorem~\ref{mainforC} $\Rightarrow$ Theorems~\ref{mainforL} and \ref{coincide}}
\label{S:Thm 1.4=>1.1 and 1.3}
For a finite non-trivial character $\chi$ with coordinates $t_{j, \chi} = \chi(c_j^*)-1$, we know that
\begin{align*}
\min\big\{ \val_q(t_{1,\chi}) , \dots, \val_q(t_{\ell, \chi}) \big\} & = \min\big\{ \val_q(\chi(c_1^*)-1) , \dots, \val_q(\chi(c_\ell^*)-1) \big\} 
\\
&= \frac 1a \cdot \frac 1{p^{m_\chi-1}(p-1)}.
\end{align*}
Hence Theorem \ref{mainforC}(1) implies
\begin{equation}
\label{E:valuation of wk}
\val_q(w_k(\chi)) \geq \lambda_k \cdot  \frac 1a \cdot \frac 1{p^{m_\chi-1}(p-1)} = \frac{k(k-1)}{2dp^{m_\chi-1}}.
\end{equation}
Moreover, by Lemma~\ref{L:admissible locus contains finite characters}, the point corresponding to this finite character $\chi$ lies on the admissible locus $\calW^\adm$.
So Theorem~\ref{mainforC}(2) implies that the equality in \eqref{E:valuation of wk} holds for $k=nd$ or $nd+1$.

From this, we deduce that the $q$-adic Newton polygon of $C_f^*(\chi,s)$ lies above the polygon with vertices $\big(k, \frac{k(k-1)}{2dp^{m_\chi-1}}\big)$, and so it must pass through the points $\big(nd, \frac{n(nd-1)}{2p^{m_\chi-1}}\big)$ and $\big(nd+1, \frac{n(nd+1)}{2p^{m_\chi-1}}\big)$ given by $(k, \val_q(w_k(\chi)))$ for $k=nd$ and $nd+1$ with $n \in \ZZ_{\geq 0}$.
This completes the proof of Theorem~\ref{coincide}.

For Theorem~\ref{mainforL}, we observe that  $L^*_f(\chi, s) = \frac{C^*_f(\chi,s)}{C^*_f(\chi,qs)}$ is a polynomial of degree $dp^{m_\chi-1}$, and the set 
$\big\{\alpha\in \CC_p\;|\;\alpha^{-1}\ \mathrm{is\ a\ root\ of}\ L^*_f(\chi, s)=0\big\}$ is the same as the set 
$$\big\{\beta\in \CC_p\;|\;\beta^{-1}\ \mathrm{is\ a\ root\ of}\ C^*_f(\chi, s)=0 \textrm{ and }\val_q(\beta)\in [0,1)\big\}.$$
So Theorem~\ref{mainforL} follows from Theorem~\ref{coincide} directly as $L_f(\chi, s)$ is obtained from $L_f^*(\chi, s)$ by removing its 
unique linear factor with slope zero. \hfill $\Box $

\begin{remark}
\label{R:all cont char okay}
The same argument proves the analog of Theorem~\ref{coincide} for all continuous characters $\chi$ of $\ZZ_{p^\ell}$ whose corresponding points on the weight space lie in $\calW^\adm$.
\end{remark}

\section{$I$-adic exponential sums}\label{section 2}

We fix the polynomial $\bar f$ and its Teichm\"uller lift $f$ as in the introduction.

\begin{notation}
We first recall that the \emph{Artin--Hasse exponential series} is defined by
\begin{equation}\label{Artin-Hasse}
	E(\pi) = \exp\big( \sum_{i=0}^\infty \frac{\pi^{p^i}}{p^i} \big) = \prod\limits_{p \nmid i,\ i \geq 1} \big( 1-\pi^i\big)^{-\mu(i)/i} \in 1+ \pi + \pi^2 \ZZ_p[\![ \pi ]\!].
\end{equation}
Setting $T = E(\pi) -1$ defines an isomorphism $\ZZ_p\llbracket \pi \rrbracket \cong \ZZ_p\llbracket T\rrbracket$.

For the rest of the paper, it will be more convenient to set $T_j = E(\pi_j) -1$ for each $j$ and use $\pi_1, \dots, \pi_\ell$ as the parameters for the ring $\ZZ_p\llbracket \underline T\rrbracket \cong \ZZ_p\llbracket \underline \pi \rrbracket $.
In particular, we have
\[
I = (T_1, \dots, T_\ell) = (\pi_1, \dots, \pi_\ell).
\]
\end{notation}

\begin{definition}	
For a positive integer $k$, the \emph{I-adic exponential sum} of $f$ over $\FF_{q^k}^\times$ is\footnote{This sum agrees with $S_f(k, T)$ in \cite{liu-wan} (in the one-dimensional case).}
\[
S^*(k, \underline{\pi}): = \sum_{x \in \FF_{q^k}^\times} \prod\limits_{j=1}^\ell E(\pi_j)^{\Tr_{\QQ_{q^k} / \QQ_p}[c_jf(\omega( x))]} \in \ZZ_p[\![ \underline{\pi}]\!].
\]
Note that the sum is taken over $\FF_{q^k}^\times$. 
The superscript $*$ reminds us that we are working over the torus $\GG_m$. 
We define the \emph{$I$-adic characteristic power series associated to $f$} to be\footnote{Our $C_f^*(\underline{\pi},s)$ agrees with the $C_f(T,s)$ in \cite{liu-wan} (in the one-dimensional case); we will not introduce a version $C_f(T,s)$ without the star since it will not be used in our proof.}
\begin{eqnarray}
\label{E:Cfstar}
C_f^*(\underline{\pi},s) &:= &\exp \Big( \sum_{k=1}^\infty \frac 1{1-q^k} S^*(k,\underline{\pi})\frac{s^k}{k} \Big) 
\\
\nonumber
&=& \displaystyle \sum_{k=0}^\infty w_k(\underline \pi) s^k \in \ZZ_p\llbracket \underline \pi, s\rrbracket.
\end{eqnarray}

The \emph{$I$-adic L-series of $f$} is defined by 
$$
L_f^*(\underline{\pi},s) = \exp \Big( \sum_{k=1}^\infty S^*(k,\underline{\pi})\frac{s^k} {k} \Big).$$
These two series determine each other, and are related by the formula 
$$C_f^*(\underline{\pi},s) = L_f^*(\underline{\pi},s)L_f^*(\underline{\pi},qs)L_f^*(\underline{\pi},q^2s) \cdots. $$
It is clear that for a finite character $\chi: \ZZ_{p^\ell} \to \CC_p^\times$, 
\[
L_f^*(\chi,s) =L_f^*(\underline{\pi},s)|_{E(\pi_j) = \chi(c_j^*) \textrm{ for all }j}, \ \ 
C_f^*(\chi,s) =C_f^*(\underline{\pi},s)|_{E(\pi_j) = \chi(c_j^*) \textrm{ for all }j}.
\]
Here the subscripts mean to evaluate the power series at $\pi_j \in \gothm_{\CC_p}$ for which $E(\pi_j) = \chi(c_j^*)$ (the elements $c_j^*$ are defined just before Theorem~\ref{mainforC}).

\end{definition}
\begin{hypothesis}
\label{H:teichmuller}
From now till the end of Section~\ref{Sec:main proof}, assume the chosen basis $\{c_1, \dots, c_\ell\}$ consists of Teichm\"uller lifts, i.e. $c_j = \omega(\bar c_j)$ for $j =1, \dots, \ell$.
\end{hypothesis}

\begin{notation}
For our given polynomial $f(x) = \sum\limits_{i=0}^d a_i x^i \in \ZZ_q[x]$, we put
\begin{equation}
\label{E:Ef(x)}
E_f(x)_\pi := \prod\limits_{i=0}^d E(a_i \pi x^i) \in \ZZ_q[\![\pi]\!] [\![ x ]\!].\end{equation}
So $E_{c_j f}(x)_{\pi_j}$ would mean $\prod\limits_{i=0}^d E(c_ja_i \pi_jx^i)$.
If $\sigma$ denotes the arithmetic $p$-Frobenius automorphism which acts naturally on $\QQ_q$, and  trivially on $\pi$ and $x$, then we have,  for every $j \in \ZZ_{\geq 0}$, 
\[
E_f^{\sigma^j}(x)_\pi = \prod\limits_{i=0}^d E(a_i^{\sigma^j} \pi x^i) \in \ZZ_q[\![\pi]\!] [\![ x ]\!].\]
\end{notation}

\begin{lemma}
\label{L:estimate of Ef(x)}
\emph{(1)} If we write $E_f(x)_\pi =  \sum\limits_{n=0}^\infty b_n(\pi)x^n \in \ZZ_q\llbracket \pi\rrbracket \llbracket x \rrbracket$, 
then $b_n(\pi) \in \pi^{\lceil n/d\rceil} \ZZ_q\llbracket \pi \rrbracket$.

\emph{(2)}
If we write 
$
\prod\limits_{j=1}^{\ell} E_{ c_jf}(x)_{\pi_j} =
\sum\limits_{n=0}^\infty e_n(\underline \pi) x^n \in \ZZ_q\llbracket \underline \pi \rrbracket \llbracket x \rrbracket,
$
then
$
e_n(\underline \pi) \in I^{\lceil n/d\rceil}
$ and $e_0=1$.
\end{lemma}

\begin{proof}
Note that the $i$th factor of $E_f(x)_\pi$ in \eqref{E:Ef(x)} is a power series in $\pi x^i$ for $1 \leq i \leq d$; so every term in their product is a sum of products of $\pi, \pi x, \dots, \pi x^d$. (1) is  clear from this.  (2) follows from (1) immediately.
\end{proof}

\begin{convention}
\label{Conv:matrices start with zero}
In this paper, the row and column indices of matrices start with zero.
\end{convention}

\subsection{Dwork's trace formula}
Consider the following ``Banach module" over $\ZZ_q\llbracket \underline \pi\rrbracket$ with ``orthonormal basis" $\Gamma: = \{1, x, x^2, \dots\}$:\footnote{Since $\ZZ_q\llbracket \underline \pi\rrbracket$ is not a Banach algebra, $\widetilde \bfB$ is not a Banach space in the literal sense.}
\begin{equation}
\label{E:widetilde B}
\widetilde{\textbf{B}}:=\ZZ_q \llbracket \underline \pi \rrbracket \langle x\rangle = \Big\{\sum\limits_{n=0}^{\infty} d_{n}(\underline \pi) x^n\; |\; d_n(\underline \pi) \in \ZZ_q[\![ \underline{\pi}]\!]\; \textrm{ and } \lim\limits_{n\to \infty} d_{n}=0\Big\}.\footnote{This $\widetilde \bfB$ is different from the space $B$ considered in \cite[Section~2]{Davis-wan-xiao}, where the extra rescaling factors $\pi^{i/d}$ are used to simplify the notation of the proof. We cannot do such simplification over a multi-dimensional weight space.}
\end{equation}

Let $\psi_p$ denote the operator on $\widetilde{\bold{B}}$ defined by 
$$\psi_p\Big(\sum_{n\geq 0}^\infty d_n(\underline \pi) x^n\Big): = \sum_{n\geq 0}^\infty d_{pn}(\underline \pi) x^n,$$ 
and let $\psi$  be the composite linear operator 
\begin{equation}
\label{E:psi}
\psi := \psi_p \circ\prod\limits_{j=1}^{\ell} E_{ c_jf}(x)_{\pi_j}: \widetilde{\bold{B}} \longrightarrow \widetilde{\bold{B}},
\end{equation}
where $ \prod\limits_{j=1}^{\ell} E_{ c_jf}(x)_{\pi_j}(g):=\prod\limits_{j=1}^{\ell} E_{ c_jf}(x)_{\pi_j}\cdot g$ for any $g\in \widetilde{\bold{B}}$. 
One can easily check that 
\[
\psi \big(x^n\big) = \sum_{m=0}^\infty e_{mp-n}(\underline \pi) x^m,
\]
where $e_n = e_n(\underline \pi)$ is as defined in Lemma~\ref{L:estimate of Ef(x)}(2) (for $i < 0$, we set $e_i = 0$).
Explicitly, the matrix of $\psi$ with respect to the basis $\Gamma:=\{1,x,x^2,\dots\}$ is given by
\begin{equation}
\label{E:explicit N}
N=\big( e_{mp-n}\big)_{m,n\geq 0} =\begin{pmatrix} 
e_0&0&\cdots&0& 0 &\cdots &0&\cdots\\
e_p & e_{p-1} & \cdots & e_0 &  0  & \cdots  & 0 & \cdots\\
e_{2p} & e_{2p-1} & \cdots & e_p & e_{p-1}  & \cdots & e_0 & \cdots\\
\vdots & \vdots & \ddots & \vdots & \vdots & \vdots & \ddots  & \ddots\\
e_{mp} & e_{mp-1} & \cdots & e_{mp-p} & e_{mp-p-1} & \cdots & e_{mp-2p} & \cdots\\
\vdots & \vdots & \ddots & \vdots & \vdots & \ddots & \vdots & \ddots
\end{pmatrix}.
\end{equation}

The operator $\sigma^{-1}\circ \psi$ is $\sigma^{-1}$-linear, but its $a$-th iteration $(\sigma^{-1}\circ \psi)^a$ is linear since 
$\sigma^a$ acts trivially on $\ZZ_q\llbracket \underline \pi\rrbracket$. For the same reason, $\sigma^a(N) =N$. 

\begin{theorem}[Dwork Trace Formula]

For every $k>0$, we have
$$S^*(k,\underline{\pi})=(q^k-1)\Tr_{\widetilde{\bold{B}}/\ZZ_q[\![ \underline{\pi}]\!]}\big((\sigma^{-1}\circ\psi)^{ak}\big).$$
\end{theorem}
\begin{proof} The proof is the same as in \cite[Lemma~4.7]{liu-wan}. The key point is that the Dwork trace formula is 
universally true, see \cite{wan96} for a thorough understanding of the universal Dwork trace formula. 
\end{proof}

\begin{corollary}\label{determinant}
The theorem above has an equivalent multiplicative form:
\begin{equation}\label{dwork}
C_f^*(\underline{\pi}, s)= \det\big(I-s \sigma^{a-1}(N) \cdots \sigma(N) N \big).
\end{equation}
\end{corollary}

\begin{proof}

It follows from the following list of equalities 
\begin{align*}
C_f^*(\underline{\pi}, s) =& \exp \big( \sum_{k=1}^\infty \frac 1{1-q^k} S^*(k,\underline{\pi})\frac{s^k}{k} \big)  \\
=&\exp \big( \sum_{k=1}^\infty {-\Tr_{\widetilde{\textbf{B}}/\ZZ_q[\![ \underline{\pi}]\!]}((\sigma^{-1}\circ\psi)^{ak})}\frac{s^k}{k} \big) \\
=& \det  \big(I-(\sigma^{-1}\circ\psi)^{a}s\;\big|\;\widetilde{\bold{B}}\big)\\
=& \det\big(I-s \sigma^{-1}(N) \sigma^{-2}(N) \cdots \sigma^{-a}(N)\big)
\\
=& \det \big(I-s \sigma^{a-1}(N) \cdots \sigma(N)N\big). \qedhere
\end{align*}
\end{proof}

\section{A Hodge bound for $C_f^*($\underline{\it T}$,s)$}
\label{Sec:Hodge bound}

In this section, we prove Theorem~\ref{mainforC}(1), which will follow from the key estimate of a certain (variant of) Hodge polygon  bound in Proposition~\ref{lemma}.
We continue to assume Hypothesis~\ref{H:teichmuller}.

\begin{notation}
The ideal in $\ZZ_q[\![ \underline{\pi}]\!]$ generated by $(\pi_1,..., \pi_{\ell})$ is also denoted by $I$. 
We define a \emph{valuation function}
\[\val_{ I}:\ZZ_q[\![ \underline{\pi}]\!] \xrightarrow{\qquad} \ZZ \cup \{\infty\},\qquad\qquad
\]
\[
\val_I(x) = \begin{cases}
n & \textrm{ if }x \in I^n \textrm{ and } x \notin I^{n+1},\\
\infty & \textrm{ if }x=0.
\end{cases}
\]	
Note that $\val_{ I}(ab) = \val_{ I}(a) + \val_{ I}(b)$ for $a,b \in \ZZ_q\llbracket \underline \pi\rrbracket$.

\end{notation}

\begin{remark}
Using this valuation function, we can similarly define the \emph{$I$-adic Newton polygon} of a power series $\sum\limits_{k\geq 0}^\infty c_k(\underline \pi) s^k \in \ZZ_q\llbracket \underline \pi, s \rrbracket$ to be the lower convex hull of the points $(k, \val_I(c_k(\underline \pi)))$. Then Theorem~\ref{mainforC} says that the $I$-adic Newton polygon of $C_f^*(\underline \pi, s)$ lies above the polygon with vertices $(k, \lambda_k)$ with $\lambda_k = \frac{ak(k-1)(p-1)}{2d}$, and it passes through the points $(nd, \lambda_{nd})$ and $(nd+1, \lambda_{nd+1})$ for all $n \in \ZZ_{\geq 0}$.
\end{remark}

\begin{definition}
\label{D:twisted incremental}
Let $M_\infty(\ZZ_q\llbracket \underline \pi \rrbracket)$ denote the set of matrices with entries in $\ZZ_q\llbracket \underline \pi \rrbracket$, whose rows and columns are indexed by $\ZZ_{\geq 0}$ (recall from Convention~\ref{Conv:matrices start with zero} that all row and column indices start from $0$).

We say a matrix $N=(h_{mn})_{m,n\geq 0}\in M_\infty(\ZZ_q[\![\underline{\pi}]\!])$ is \emph{twisted $I$-adically incremental (in $d$ steps)} if $\val_I(h_{mn})\geq \frac{mp-n}d$ (or equivalently $\val_I(h_{mn})\geq \lceil\frac{mp-n}d\rceil$) for all integers $m,n\geq 0$.\footnote{We invite the readers to compare this with \cite[Proposition~3.12(1)]{liu-wan-xiao}, which is the estimate before the conjugation by a diagonal matrix.}
By Lemma~\ref{L:estimate of Ef(x)}(2) and \eqref{E:explicit N}, we see that the matrix $N$ (and more generally $\sigma^i(N)$) is twisted $I$-adically incremental for every $i$.
\end{definition}

Proposition~\ref{lemma} below allows us to control the $I$-adic Newton polygon of $C_f^*(\underline \pi, s)$ using the twisted $I$-adic incremental property of these $\sigma^i(N)$'s.

\begin{notation}
For a matrix $M$, we write \[\left[ \begin{array}{cccccccccc}
m_0 & m_1 &\cdots&m_{k-1} \\
n_0 & n_1 &\cdots&n_{k-1}  \end{array} \right]_M\]
for the $k\times k$-matrix formed by elements whose row indices belong to $\{m_0,m_1,\dots,m_{k-1}\}$ and whose column indices belong to $\{n_0,n_1,\dots,n_{k-1}\}$.
\end{notation}

\begin{lemma}
\label{L:twisted Hodge bound to det}
Let $M = (h_{mn}) \in M_\infty(\ZZ_q\llbracket \underline \pi\rrbracket)$ be a twisted $I$-adically incremental matrix, then for indices $m_0, \dots, m_{k-1}$ and $n_0, \dots, n_{k-1}$, we have
\[
\val_I \Big( \det \left[ \begin{array}{cccccccccc}
m_0 & m_1 &\cdots&m_{k-1} \\
n_0 & n_1 &\cdots&n_{k-1}  \end{array} \right]_M \Big) \geq \sum_{i = 0}^{k-1} \frac{pm_i-n_i}{d}.
\]
\end{lemma}
\begin{proof}
In fact, we show that the $\val_I$ of each term in the determinant above is greater than or equal to $\sum\limits_{i = 0}^{k-1} \frac{pm_i-n_i}{d}.$
Indeed, for each permutation $\sigma \in \mathrm{Aut}(\{0, \dots, k-1\})$, we have
\[
\val_I\big(h_{m_0 n_{\sigma(0)}} \cdots h_{m_{k-1} n_{\sigma(k-1)}} \big) \geq \frac{pm_0 - n_{\sigma(0)}}d + \cdots + \frac{pm_{k-1} - n_{\sigma(k-1)}}d \geq \sum_{i = 0}^{k-1} \frac{pm_i-n_i}{d}.
\]
The lemma follows.
\end{proof}

\begin{proposition}\label{lemma}

Let $M_0,M_1,\dots,M_{a-1} \in M_\infty(\ZZ_q\llbracket \underline \pi \rrbracket)$ be twisted $I$-adically incremental matrices, and let $\det(I-sM_{a-1} \cdots M_1 M_0)=\sum\limits_{k=0}^\infty (-1)^k r_k(\underline \pi) s^k $ denote the characteristic power series of their product, then for every integer $k\geq 0$, we have 
\[\val_I(r_k(\underline \pi))\geq  \frac{ak(k-1)(p-1)}{2d}, \quad \textrm{and}\]
\[
r_k(\underline \pi)\equiv\prod\limits_{j=0}^{a-1}\bigg(\det \left[ \begin{array}{cccccccccc}
0 &1 &\cdots&k-1
\\
0 &1 &\cdots&k-1 \end{array} \right]_{M_j}\bigg)  \quad \bmod \ {I^{\lceil\frac{ak(k-1)(p-1)+(p-1)}{2d}\rceil}}.\]

\end{proposition}
\begin{proof}
From the definition of characteristic power series, we see 
 \begin{equation}
 \label{E:expression of char power series}
\begin{split}
r_k(\underline \pi)& =\sum\limits_{0\leq m_0<m_1<\cdots<m_{k-1}<\infty}\det \left[ \begin{array}{cccccccccc}
m_0 & m_1 &\cdots&m_{k-1} \\
m_0 & m_1 &\cdots&m_{k-1}  \end{array} \right]_{M_{a-1}\cdots M_1M_0}     \\
&=\sum_{\substack{0\leq m_{0,0}<m_{0,1}<\cdots<m_{0,k-1}<\infty \\ \cdots\\0\leq m_{a-1,0}<m_{a-1,1}<\cdots<m_{a-1,k-1}<\infty}} \det\bigg(\prod\limits_{j=0}^{a-1}\left[ \begin{array}{cccccccccc}
m_{j+1,0} & m_{j+1,1} &\cdots&m_{j+1,k-1} \\
m_{j,0} & m_{j,1} &\cdots&m_{j,k-1}  \end{array} \right]_{M_j}\bigg)            \\
 &=\sum_{\substack{0\leq m_{0,0}<m_{0,1}<\cdots<m_{0,k-1}<\infty \\ \cdots\\0\leq m_{a-1,0}<m_{a-1,1}<\cdots<m_{a-1,k-1}<\infty}} \prod\limits_{j=0}^{a-1}\bigg(
 \det \left[ \begin{array}{cccccccccc}
m_{j+1,0} & m_{j+1,1} &\cdots&m_{j+1,k-1} \\
m_{j,0} & m_{j,1} &\cdots&m_{j,k-1} \end{array} \right]_{M_j}\bigg).
\end{split}
\end{equation}
Here and after, we set $m_{a,i}=m_{0,i}$ for all $0\leq i\leq k-1$.

Since every $M_i$ is twisted $I$-adically incremental, we can control each term in \eqref{E:expression of char power series} using Lemma~\ref{L:twisted Hodge bound to det}:
\begin{equation}
\label{E:valuation of each small term}
\begin{split}
&\val_{I}\Bigg(\prod\limits_{j=0}^{a-1}\bigg(
 \det \left[ \begin{array}{cccccccccc}
m_{j+1,0} & m_{j+1,1} &\cdots&m_{j+1,k-1} \\
m_{j,0} & m_{j,1} &\cdots&m_{j,k-1}  \end{array} \right]_{M_j}\bigg)\Bigg)       \\
\geq& \sum\limits^{a-1}_{j=0}\sum\limits^{k-1}_{i=0}\frac{pm_{j+1,i}-m_{j,i}}d  =  \frac{p-1}d\sum\limits^{a-1}_{j=0}\sum\limits^{k-1}_{i=0}m_{j,i} \geq\frac{ak(k-1)(p-1)}{2d}.
\end{split}
\end{equation}
This verifies the first statement.

Notice that the last inequality of \eqref{E:valuation of each small term} is an equality if and only if $m_{j,i}=i$ for all $ 0\leq j\leq a-1$ and $0\leq i\leq k-1$; and when it is not an equality, \eqref{E:valuation of each small term} is greater than or equal to $\frac{ak(k-1)(p-1)+(p-1)}{2d}$. Therefore, we have
\[
r_k(\underline \pi)\equiv\prod\limits_{j=0}^{a-1}\bigg(\det \left[ \begin{array}{cccccccccc}
0 &1 &\cdots&k-1
\\
0 &1 &\cdots&k-1 \end{array} \right]_{M_j}\bigg)  \quad \bmod \ {I^{\lceil\frac{ak(k-1)(p-1)+(p-1)}{2d}\rceil}}.\qedhere\]
\end{proof}

\medskip
\begin{proof}[Proof of Theorem~\ref{mainforC}(1)]
By Corollary \ref{determinant}, $C_f^*(\underline \pi, s)$ is the characteristic power series of the product $\sigma^{a-1}(N)\cdots \sigma(N)N$.
But each $\sigma^i(N)$ is twisted $I$-adically incremental, which implies Theorem~\ref{mainforC}(1) by applying Proposition \ref{lemma}.\footnote{Although this section assumes Hypothesis~\ref{H:teichmuller}, as pointed out in Remark~\ref{R:independence}, the validity of Theorem~\ref{mainforC}(1) does not depend on the choice of the basis $\{c_1, \dots, c_\ell\}$.
So our proof is complete.}
\end{proof}

\section{The proof of Theorem \ref{mainforC}(2)}
\label{Sec:main proof}

As a reminder, Hypothesis~\ref{H:teichmuller} is still in force in this section.
This section is devoted to prove Theorem \ref{mainforC}(2), whose proof will appear at the end of this section.
Its key ingredient is the following.
\begin{theorem}
\label{T:key technical theorem}
Put $\bar \gothT:=\sum\limits_{j=1}^\ell \bar c_j\pi_j \in \FF_q\llbracket \underline \pi \rrbracket$, then
\[
\det \left[ \begin{array}{cccccccccc}
0 & 1 &\cdots&k-1 \\
0 & 1 &\cdots&k-1  \end{array} \right]_N  \bmod p
\]
viewed as an element of $\FF_q \llbracket \underline \pi \rrbracket$, lies in $\FF_q\llbracket \bar \gothT\rrbracket$.
Moreover, the coefficients of this determinant as a power series in $\bar \gothT$ does not depend on $\ell$.
\end{theorem}

\begin{proof}
We write $\bar e_n: = e_n \bmod p \in \FF_q\llbracket \underline \pi \rrbracket$.
Consider the following $(kp-p+1)\times (kp-p+1)$ matrix
\[\bar N_{k}^\mathrm{T}
=\begin{tikzpicture}[baseline,decoration=brace]
\matrix (m) [mymatrix] {
\bar e_0&\bar e_{p}&\bar e_{2p}&\cdots&\bar e_{(k-1)p}&0&0&0&\cdots&0\\
0 &\bar  e_{p-1} &\bar e_{2p-1}&\cdots& \bar e_{(k-1)p-1}&0&0&0&\cdots&0\\
\vdots&\vdots&\vdots&\ddots&\vdots&\vdots&\vdots&\vdots&\ddots&\vdots\\
0&\bar e_{0}&\bar e_p&\cdots&\bar e_{(k-1)p-p}&0&0&0&\cdots&0\\
0  &  0  &\bar e_{p-1}&\cdots&\bar e_{(k-1)p-p-1} &0&0&0&\cdots&0\\
\vdots &\vdots &  \vdots &\ddots &\vdots&\vdots&\vdots&\vdots&\ddots&\vdots&\\
0  &  0  &*&\cdots &\bar e_{(k-1)(p-1)}&0&0&0&\cdots&0\\
	0  &  0  &*&\cdots &\bar e_{(k-1)(p-1)-1}&1&0&0&\cdots&0\\
	0  &  0  &*&\cdots &\bar e_{(k-1)(p-1)-2}&0&1&0&\cdots&0\\
	0  &  0  &*&\cdots &\bar e_{(k-1)(p-1)-3}&0&0&1&\cdots&0\\
\vdots  & \vdots  &\vdots&\ddots &\vdots&\vdots&\vdots&\vdots&\ddots&\vdots\\
	0  &  0  &0&\cdots &\bar e_0&0&0&0&0&1\\
};
\draw[decorate,transform canvas={xshift=-1.4em},thick] (m-7-1.south west) -- node[left=2pt] {$k$} (m-1-1.north west);
\draw[decorate,transform canvas={yshift=0.5em},thick] (m-1-1.north west) -- node[above=2pt] {$k$} (m-1-5.north east);
\draw[decorate,transform canvas={yshift=0.5em},thick] (m-1-6.north west) -- node[above=2pt] {$(k-1)(p-1)$} (m-1-10.north east);
\end{tikzpicture}
.\]
Note that the upper left $k \times k$-block of $\bar N_k^\mathrm{T}$ is the transpose of $\left[ \begin{array}{cccccccccc}
0 & 1 &\cdots&k-1 \\
0 & 1 &\cdots&k-1 \end{array} \right]_N$ modulo $p$,\footnote{Here, we made a tough choice to consider the transpose instead, so that the display of $N_{k}^\mathrm{T}$ is much nicer.} so we have an equality in $\FF_q \llbracket \underline \pi \rrbracket $:
\begin{equation} \label{N to Nnd}
\det \left[ \begin{array}{cccccccccc}
0 & 1 &\cdots&k-1 \\
0 & 1 &\cdots&k-1 \end{array} \right]_N \bmod p=\det(\bar N_k^\mathrm{T}).
\end{equation}

To study $\bar N_k^\rmT$, we need the following.

\begin{lemma}\label{induction}
We have the following equality and congruence. \begin{align}
\label{E:iteration equality}
ne_n&=\sum\limits^d_{i=1}\sum\limits^\infty_{r=0}i \cdot e_{n-ip^r}
 a_i^{p^r}\big(\sum\limits_{j=1}^\ell (c_j\pi_j)^{p^r}\big)
\\
\label{E:congruence}
&\equiv\sum\limits^d_{i=1}\sum\limits^\infty_{r=0}i \cdot e_{n-ip^r} a_i^{p^r}\big(\sum\limits_{j=1}^\ell c_j\pi_j\big)^{p^r}\pmod{p}.
\end{align}
\end{lemma}

\begin{proof}
Taking the derivative of $\prod\limits_{j=1}^\ell E_{ c_jf}(x)_{\pi_j}$ gives
\[\Big(\prod\limits_{j=1}^\ell E_{ c_jf}(x)_{\pi_j}\Big)'=\Big(\prod\limits_{j=1}^\ell E_{ c_jf}(x)_{\pi_j}\Big) \Big(\sum\limits^d_{i=1}\sum\limits^\infty_{r=0} \Big(\sum\limits_{j=1}^\ell (c_j\pi_j)^{p^r}\Big) ix^{i p^r-1}  a_i^{p^r}\Big).\]
Replacing $\prod\limits_{j=1}^\ell E_{ c_jf}(x)_{\pi_j}$ by $\sum\limits_{n=0}^\infty e_nx^n$, the above equality becomes
\[\sum\limits^\infty_{n=0}ne_nx^{n-1}=\Big(\sum\limits_{n=0}^\infty e_nx^n\Big)\Big(\sum\limits^d_{i=1}\sum\limits^\infty_{r=0}\Big(\sum\limits_{j=1}^\ell (c_j\pi_j)^{p^r}\Big)ix^{i p^r-1}  a_i^{p^r}\Big).\]
Then \eqref{E:iteration equality} follows by comparing the $x^{n-1}$-coefficients.
The congruence \eqref{E:congruence} follows from the easy fact that $\big(\sum\limits_{j=1}^\ell c_j\pi_j\big)^{p^r} \equiv \sum\limits_{j=1}^\ell (c_j\pi_j)^{p^r} \pmod p$.
\end{proof}

We now continue with the proof of Theorem~\ref{T:key technical theorem}.
Let $\bar N_{k,1}^\rmT$ be the matrix consisting of the first $k$ columns of $\bar N_k^\rmT$. Then the $(m,n)$-entry of $\bar N_{k, 1}^\rmT$ is just $\bar e_{np-m}$.
Applying Lemma \ref{induction} to $np-m$ in place of $n$ (and then taking the reduction modulo $p$), we deduce
$$-m \bar e_{np-m}  = (np-m)\bar e_{np-m} = \sum\limits_{i=1}^{d}\sum\limits_{r=0}^{\infty}i \cdot \bar e_{np-(m+ip^r)}\bar a_i^{p^r}\bar \gothT^{p^r}$$
in $\FF_q\llbracket \underline \pi \rrbracket$.
Note that the coefficients in the above congruence do not involve the column index $n$.
So if we use $\bar R_m(k)$ to denote the $m$th row of $\bar N_{k, 1}^\rmT$ (and $\bar R_m(k)$ is the zero row if $m > (k-1)p$), we get  
\begin{equation}
\label{E:mod p recursive relation}
m \cdot \bar R_m(k)+\sum\limits_{i=1}^{d}\sum\limits_{r=0}^{\infty}i \cdot \bar R_{m+ip^r}(k) \bar  a_i^{p^r} \bar \gothT^{p^r} = 0
\end{equation}
for all $0\leq m\leq kp-p$.
In other words, the $m$th row of $\bar N_{k,1}^\rmT$ with $m\not\equiv 0 \pmod{p}$ can be written as a linear combination of the rows below it, and the coefficients of this linear combination belong to $\FF_q\llbracket \bar \gothT \rrbracket$ (as opposed to $\FF_q \llbracket \underline \pi \rrbracket$). 

To take advantage of this linear relation among the rows $\bar R_m(k)$, we define the (upper triangular) matrix $\bar A_{k}(\bar \gothT) \in \mathrm M_{kp-p+1}(\FF_q\llbracket \bar \gothT\rrbracket)$ so that, if we write $\bar R_m(k)'$ to denote the $m$th row of $\bar P_{k}:= \bar  A_{k}(\bar \gothT)\bar N_k^\rmT$, then we have
\begin{equation}
\bar R_m(k)'=\begin{cases}
m\bar R_m(k)+\sum\limits_{i=1}^{d}\sum\limits_{r=0}^{\infty}i \cdot \bar R_{m+ip^r}(k) \bar a_i^{p^r}\bar \gothT^{p^r},&\textrm{if } p\nmid m,\\
\bar R_m(k),& \textrm{if } p\,|\,m.
\end{cases}
\end{equation}
Explicitly, if we write $\bar A_{k}(\bar \gothT) = (\bar a_{mn})_{m,n \in \ZZ_{\geq 0}}$, we have
\[
\bar a_{mn} =  \begin{cases}
i \bar a_i^{p^r}\bar \gothT^{p^r}& \textrm{when } n-m = i p^r \textrm{ with }1 \leq i \leq d \textrm{ and } p \nmid i, \\
1 & \textrm{when }m = n \textrm{ and }p \,|\, m,\\
m & \textrm{when }m = n\textrm{ and }p \nmid m,\\
0 & \textrm{otherwise}.
\end{cases}
\]
Note that, in the first case, there is only one term, because the other terms with $p\,|\, i$ are zero modulo $p$.

According to the recurrence relations of $\{e_n\}$ in \eqref{E:mod p recursive relation}, the matrix $\bar P_{k}: = \bar A_{k}(\bar \gothT) \bar N_{k}^\rmT$ takes the following form
\[\bar P_{k}=\Big(\bar p_{mn}\Big)_{\substack{0\leq m \leq kp-p\\0\leq n\leq kp-p}}=\begin{tikzpicture}[baseline,decoration=brace]
\matrix[mymatrix] (m)  {
\bar e_0&\bar e_{p}&\bar e_{2p}&\cdots&\bar e_{(k-1)p}&*&*&*&\cdots&*\\
0 &0 &0&\cdots&0&*&*&*&\cdots&*\\
\vdots  & \vdots  &\vdots&\ddots &\vdots&\vdots&\vdots&\vdots&\ddots&\vdots\\
0&\bar e_{0}&\bar e_p&\cdots&\bar e_{(k-2)p}&*&*&*&\cdots&*\\
0  &  0  &0&\cdots&0&*&*&*&\cdots&*\\
\vdots  & \vdots  &\vdots&\ddots &\vdots&\vdots&\vdots&\vdots&\ddots&\vdots\\
0  &  0  &\bar e_0&\cdots &\bar e_{(k-3)p}&*&*&*&\cdots&*\\
0  &  0  &0&\cdots &0&*&*&*&\cdots&*\\
\vdots  & \vdots  &\vdots&\ddots &\vdots&\vdots&\vdots&\vdots&\ddots&\vdots\\
\vdots  & \vdots  &\vdots&\ddots &\vdots&\vdots&\vdots&\vdots&\ddots&\vdots\\
0  &  0  &0&\cdots & 0&*&*&*&\cdots&*\\
0  &  0  &0&\cdots &\bar e_0&*&*&*&\cdots&*\\
};
\mymatrixbracetop{1}{5}{$k$}
\mymatrixbracetop{6}{10}{$(k-1)(p-1)$}
\mymatrixbraceright{1}{9}{$k$} 
\end{tikzpicture},\]
where for $n\geq k$, $\bar p _{mn}$ is a function of $\bar \gothT$ given by
\begin{equation}
\label{E:explicit pmk}
\bar p_{mn}= \begin{cases}
i \bar a_i^{p^r}\bar \gothT^{p^r}& \textrm{when } n-m = i p^r \textrm{ with }1 \leq i \leq d \textrm{ and } p \nmid i, \\
1 & \textrm{when }m = n \textrm{ and }p \,|\, m,\\
m & \textrm{when }m = n\textrm{ and }p \nmid m,\\
0 & \textrm{otherwise}.
\end{cases}
\end{equation}

Since $\bar A_{k}(\bar \gothT)$ is upper triangular, we have \begin{equation}\label{And}
\det(\bar A_{k}(\bar \gothT))=\prod\limits_{i=1}^{p-1}i^{k-1} =  (-1)^{k-1} \quad \textrm{ in } \FF_q\llbracket \bar \gothT \rrbracket.
\end{equation} 
For a similar reason (and $e_0=1$), we have
\[
\det (\bar P_k) =  \det \left[ \begin{array}{cccccccccc}
	1 & 2 &\cdots&\lfloor \frac{i}{p}\rfloor+i&\cdots&kp-p-1
	\\
	k & k+1 &\cdots&k+i-1&\cdots&kp-p  \end{array} \right]_{\bar P_k}.
\]

Combining these two, we deduce
\begin{align*}
(-1)^{k-1}& \det (\bar N_k^\rmT)  = \det(\bar A_{k}(\bar \gothT)) \det (\bar N_k^\rmT) = 
\det(\bar P_k)
\\
& = \det \left[ \begin{array}{cccccccccc}
	1 & 2 &\cdots&\lfloor \frac{i}{p}\rfloor+i&\cdots&kp-p-1
	\\
	k & k+1 &\cdots&k+i-1&\cdots&kp-p  \end{array} \right]_{\bar P_k}
.
\end{align*}
The key observation here is that the entries of the (sub)matrix 
\[
\left[ \begin{array}{cccccccccc}
	1 & 2 &\cdots&\lfloor \frac{i}{p}\rfloor+i&\cdots&kp-p-1
	\\
	k & k+1 &\cdots&k+i-1&\cdots&kp-p  \end{array} \right]_{\bar P_k}
\]
all lie in the subring $\FF_q\llbracket \bar \gothT \rrbracket$ of $\FF_q\llbracket \underline \pi \rrbracket$, as seen in its explicit form \eqref{E:explicit pmk}. Moreover, the coefficients on these entries are independent of $\ell$.
It follows that
\begin{equation}
\label{E:bar Nnd in FT}
\det (\bar N_{k}^\rmT) \in \FF_q\llbracket \bar \gothT\rrbracket
\end{equation}
is a power series whose coefficients are independent of $\ell$.
The Theorem follows from this and the equality \eqref{N to Nnd}.
\end{proof}

Now, we deduce Theorem~\ref{mainforC}(2) from Theorem~\ref{T:key technical theorem}.

\begin{proof}[Proof of Theorem~\ref{mainforC}(2)]
For $k=nd$ or $nd+1$, we note that $\lambda'_k: = \lambda_k / a = \frac{n(nd-1)(p-1)}2$ or $\frac{n(nd+1)(p-1)}2$ are integers because $p\nmid d$.

Since $N$ is twisted $I$-adically incremental,  Lemma~\ref{L:twisted Hodge bound to det} implies that
\[
\det \left[ \begin{array}{cccccccccc}
0 & 1 &\cdots&k-1 \\
0 & 1 &\cdots&k-1  \end{array} \right]_N  \in I^{\frac{p(0+1+\cdots +(k-1)) - (0+1+\cdots +(k-1))}{d}} = I^{\lambda'_k}.
\]
Combining this with Theorem~\ref{T:key technical theorem}, we see that
\[
\det \left[ \begin{array}{cccccccccc}
0 & 1 &\cdots&k-1 \\
0 & 1 &\cdots&k-1  \end{array} \right]_N  \bmod p= \bar v_{\lambda'_k}  \bar \gothT^{\lambda'_k} + \bar v_{\lambda'_k+1}  \bar \gothT^{\lambda'_k+1} + \cdots \in\bar  \gothT^{\lambda'_k} \FF_q \llbracket \bar \gothT \rrbracket,
\]
where $\bar v_{\lambda'_k} \in \FF_q$ is independent of $\ell$.
Thus,
\begin{equation}
\label{E:congruence N}
\det \left[ \begin{array}{cccccccccc}
0 & 1 &\cdots&k-1 \\
0 & 1 &\cdots&k-1  \end{array} \right]_N  \equiv \bar v_{\lambda'_k} \bar \gothT^{\lambda'_k} \quad \bmod pI^{\lambda'_k} + I^{\lambda'_k+1}. \footnote{Here we are allowed to write $\bar v_{\lambda'_k}$ because only this element modulo $p$ affects the congruence relation.}
\end{equation}

Applying Proposition~\ref{lemma} to the series of product $\sigma^{a-1}(N)\cdots \sigma(N)N$ (whose characteristic power series defines $C_f^*(\underline \pi, s)$), we get
\[
w_k(\underline \pi) \equiv \prod_{i=0}^{a-1}
\bigg(\det \left[ \begin{array}{cccccccccc}
0 &1 &\cdots&k-1
\\
0 &1 &\cdots&k-1 \end{array} \right]_{\sigma^i(N)} \bigg)  \quad \bmod \ {I^{\lambda'_k+1}}.
\]
Combining this with \eqref{E:congruence N}, we deduce
\begin{equation}
\label{E:wk congruence}
w_k(\underline \pi) \equiv \prod_{i=0}^{a-1}\bar v_{\lambda'_k}^{p^i} \cdot \prod_{i=0}^{a-1}
\sigma^i(\bar \gothT)^{\lambda'_k} \equiv \prod_{i=0}^{a-1}\bar v_{\lambda'_k}^{p^i} \cdot \gothS(\underline T)^{\lambda_k/a} \quad \bmod  pI^{\lambda'_k} + I^{\lambda'_k+1},
\end{equation}
where the second congruence made use of the following congruence
\[
\gothS(\underline T): = \prod_{i=1}^\ell \Big(\sum_{j=1}^\ell c_j^{p^i}T_j \Big) \equiv \prod_{i=1}^\ell \Big(\sum_{j=1}^\ell c_j^{p^j}\pi_j \Big) \equiv \prod_{i=1}^\ell \sigma^i(\bar \gothT) \quad \bmod pI^\ell+I^{\ell+1}.
\]
From \eqref{E:wk congruence}, we see that Theorem~\ref{T:key technical theorem}(2) is equivalent to $\prod\limits_{i=0}^{a-1}\bar v_{\lambda'_k}^{p^i} \in \FF_p^\times$.
But as pointed out above, this element is \emph{independent} of $\ell$. We know that Theorem~\ref{mainforC}(2) holds when $\ell=1$, as proved in \cite[Propostion~3.4]{Davis-wan-xiao}, so it holds for all $\ell$.\footnote{Once again, Remark~\ref{R:independence} allows us to prove Theorem~\ref{mainforC}(2) under Hypothesis~\ref{H:teichmuller}.}
\end{proof}

\section{Artin--Schreier--Witt eigenvarieties}
\label{Sec:ASW eigenvarieties}

We now interpret Theorem~\ref{mainforC} using the language of eigenvarieties.
Recall the weight space $\calW$ 
and its admissible locus $\calW^\adm$ from Section~\ref{Sec:weight space}. We remind the readers that $\calW^\adm$ is independent of the choice of the basis $\{c_1, \dots, c_\ell\}$ (Corollary~\ref{C:admissible locus independent}) and contains all the points corresponding to finite non-trivial characters of $\ZZ_{p^\ell}$ (Lemma~\ref{L:admissible locus contains finite characters}).

The \emph{eigenvariety} $\calE_f$ associated to the Artin--Schreier--Witt tower for $\bar f(x)$ is defined as the zero locus of $C_f^*(\underline{T},s)$ inside $(\calW - \{\underline 0\}) \times \GG_{m}^\rig$, where $s$ is the coordinate of the second factor.\footnote{Here we removed the zero point of the weight space, because when $\underline T = \underline 0$, $C_f^*(0, s) = 1 -s$  is very different from other points of the weight space.}
More rigorously, for each affinoid subdomain $U = \Max(A)$ of $(\calW - \{\underline 0\}) \times \GG_{m}^\rig$,\footnote{Once again, we remind the readers that $\GG_m^\rig$ is also an increasing union of annulus $\{ s \in \CC_p\; |\; \alpha < v(s) < \beta\}$ with $\alpha, \beta \in \QQ^\times$, $\alpha$ approaching $-\infty$, and $\beta$ approaching $\infty$.} by restriction, $C_f^*(\underline T, s)$ defines a function on $U$, and $\calE_f$ over $U$ is defined to be $\Max \big(A / (C_f^*(\underline T, s))\big)$.
Gluing over an affinoid cover of $(\calW - \{\underline 0\}) \times \GG_{m}^\rig$ gives rise to $\calE_f$ as a rigid analytic subspace.

Denote the natural projection to the first factor by $\mathrm{wt}: \calE_f \to \calW - \{\underline 0 \}$; and denote the \emph{inverse} of the natural projection to the second factor by 
\[
\alpha: \calE_f \xrightarrow{\mathrm{pr}_2} \GG_{m}^\rig \xrightarrow{x \mapsto x^{-1}} \GG_{m}^\rig.
\]
We use $\calE_f^\adm: = \wt^{-1}(\calW^\adm)$ to denote the preimage of the admissible locus of the eigenvariety.

\begin{theorem}
\label{T:eigencurve theorem}
The admissible locus of the eigenvariety $\calE_f^{\adm}$ is an infinite disjoint union \[X_{0}\coprod X_{(0,1)}\coprod X_{1}\coprod X_{(1,2)}\coprod \cdots\] 
of rigid analytic spaces	such that for each interval $J = [n,n]$ or $(n, n+1)$ with $n \in \ZZ_{\geq 0}$,
\begin{itemize}
\item
the map $\wt: X_J \to \calW^\adm$ is finite and flat of degree $1$ if $J$ represents a point and of degree $d-1$ if $J$ represents a genuine interval, and
\item
for each point $x \in X_J$, we have
\[
\frac{\val_q(\alpha(x))}{\val_q(\gothS(\wt(x)))} \in \frac{a(p-1)}{\ell}\cdot J.
\]
\end{itemize}
\end{theorem}
\begin{proof}
Similar arguments have appeared multiple times in the literature; see \cite[Theorem A]{buzzard-kilford}, \cite[Theorem~1.3]{liu-wan-xiao}, or \cite[Theorem~4.2]{Davis-wan-xiao}. So we only sketch the proof here.

For a continuous character $\chi$ of $\ZZ_{p^\ell}$ whose corresponding points lies on the admissible locus $\calW^\adm$, Theorem~\ref{coincide} (see Remark~\ref{R:all cont char okay}) implies that the $q$-valuations of the zeros of $C_f^*(\chi, s)$ consists of
\begin{itemize}
\item
for all $n$, exactly one zero has valuation $-\val_q(\gothS(\chi)) \cdot \frac{an(p-1)}\ell$, and
\item
for all $n$, exactly $d-1$ zeros (counted with multiplicity) have valuations in the interval
\[
-\val_q(\gothS(\chi)) \cdot \frac{a(p-1)}\ell \cdot \big[n+\tfrac 1d, n+\tfrac {d-1}d\big].
\]
\end{itemize}
From this, we see that $\calE_f^\adm$ is the disjoint union of the following subspaces
\[
X_{[n,n]}: = \calE_f^\adm \cap \big\{ (\underline t, a_p) \in  \calW^\adm \times \GG_m^\rig\; \big|\; \val_q(a_p) = \val_q(\gothS(\underline t)) \cdot \tfrac{an(p-1)}\ell \big\}, \textrm{ and}
\]
\begin{align*}
&X_{(n,n+1)}: = \calE_f^\adm \cap \big\{ (\underline t, a_p) \in \calW^\adm  \times \GG_m^\rig\; \big|\; \val_q(a_p) \in \val_q(\gothS(\underline t)) \cdot \tfrac{a(p-1)}\ell \cdot (n,n+1) \big\}
\\
&\quad= \calE_f^\adm \cap \big\{ ( \underline t, a_p) \in  \calW^\adm \times \GG_m^\rig\; \big|\; \val_q(a_p) \in \val_q(\gothS(\underline t)) \cdot \tfrac{a(p-1)}\ell \cdot [n+\tfrac 1d,n+\tfrac{d-1}d] \big\}.
\end{align*}
Note that, restricting to every open affinoid subdomain of $\calE_f^\adm$, the above decomposition is a union of affinoid subdomains. So $\calE_f^\adm = \coprod\limits_{n=0}^\infty \big( X_{[n,n]} \sqcup X_{(n,n+1)} \big)$ is a decomposition into an infinite disjoint union of rigid subspaces. The degree of each $X_J$ follows from the description of the number of zeros above.
\end{proof}


\appendix
\section{Errata for \cite{Davis-wan-xiao}}

$\bullet$ On the lower half of page 1458, the displayed formula
\[
E_f(x) = \sum_{j=0}^\infty u_j \pi^{j/d}x^j \in B, \quad \textrm{for } u_j \in \ZZ_p.
\]
should have $u_j \in \ZZ_p\llbracket \pi^{1/d}\rrbracket$ instead.

$\bullet$ (pointed out to us by Hui June Zhu) {\bf Theorem 3.8}  on Line 2 of the second paragraph of its proof, we took $\lambda'_i$ to be the minimal integer satisfying certain properties. There might not be such $\lambda'_i$, in which case we should simply take $\lambda'_i$ to be infinity. This will not affect the proof, as all we care are those $\lambda'_i$'s that are ``close" to the lower bound polygon.

$\bullet$ {\bf Theorem~4.2(2)} The statement that each $\calC_{f,i}$ is finite and flat over $\calW$ is not literally true because $\calC_{f,i}$ often misses the point over $T=0$ in $\calW$, as the slopes at points on $\calC_{f,i}$ tend to $\infty$ as $T$ approaches to $0$. So one should replace the $\calW$ and $\calC_f$ in the statement with $\calW^\circ : = \calW- \{0\}$ and $\calC_f^\circ: = \calC_f - \mathrm{wt}^{-1}(0)$.

\end{document}